\documentclass[10.9pt,a4paper]{amsart}
\usepackage{a4wide}

\usepackage[utf8]{inputenc}
\usepackage[english]{babel}

\usepackage{amsfonts,amssymb,amsmath}

\usepackage{xcolor}
\usepackage[colorlinks,
    linkcolor={red!50!black},
    citecolor={blue!50!black},
    urlcolor={blue!80!black}]{hyperref}

\usepackage{tikz}
\usepackage[all]{xy}
\usepackage{enumitem}
\makeatletter
\newcommand{\mylabel}[2]{#2\def\@currentlabel{#2}\label{#1}}
\usetikzlibrary{calc, matrix, arrows, cd}

\newcommand{\bsm}{\left(\begin{smallmatrix}}
\newcommand{\esm}{\end{smallmatrix}\right)}

\newtheorem{theorem}{Theorem}[section]

\newtheorem{corollary}[theorem]{Corollary}

\newtheorem*{thmintro}{Theorem}
\newtheorem{question}[theorem]{Question}

\theoremstyle{definition}
\newtheorem{definition}[theorem]{Definition}

\newtheorem{example}[theorem]{Example}

\newtheorem{remark}[theorem]{Remark}

\newtheorem*{claim*}{Claim}

\newcommand{\Z}{\mathbb{Z}}

\newcommand{\R}{\mathbb{R}}

\newcommand{\id}{\operatorname{id}}

\begin{document}
\title{Homotopy ribbon discs with a fixed group}
\author{Anthony Conway}
\address{Massachusetts Institute of Technology, Cambridge MA 02139, United States}
\email{anthonyyconway@gmail.com}
\begin{abstract}
In the topological category, the classification of homotopy ribbon discs is known when the fundamental group $G$ of the exterior is~$\Z$ and the Baumslag-Solitar group~$BS(1,2)$.
We prove that if a group~$G$ is geometrically $2$-dimensional and satisfies the Farrell-Jones conjecture,  then a condition involving the fundamental group ensures that
exteriors of aspherical homotopy ribbon discs with fundamental group $G$ are s-cobordant rel.\ boundary.
When $G$ is good, this leads to the classification of such discs.
As an application,  for any knot~$J \subset S^3$ whose knot group~$G(J)$ is good,  we classify the homotopy ribbon discs for~$J \# -J$ whose complement has group~$G(J)$.
A similar application is obtained for $BS(m,n)$ when $|m-n|=1$.
\end{abstract}
\maketitle

\section{introduction}

Given a knot $K \subset S^3$, we consider the problem of classifying locally flat discs~$D \subset D^4$ with boundary~$K$, up to topological ambient isotopy rel.\ boundary.
Naturally,~$K$ need not bound such a disc (i.e.~$K$ need not be \emph{slice}) but if it does, then it is conjectured that it necessarily bounds one for which the inclusion induced map~$\pi_1(S^3 \setminus K) \to \pi_1(D^4 \setminus D)$ is surjective; such discs are called \emph{homotopy ribbon}.
For this reason, and for technical purposes, we restrict our attention to homotopy ribbon discs with boundary $K$.
Additionally, observe that if~$D_1$ and~$D_2$ are two ambiently isotopic slice discs with boundary~$K$,  then their groups must be isomorphic:~$\pi_1(D^4 \setminus D_1) \cong \pi_1(D^4 \setminus D_2)$.
The goal of this article is to study the following question. 

\begin{question}
\label{question}
Given a knot~$K \subset S^3$ and a ribbon group~$G$,  can one describe the set 
of homotopy ribbon discs for~$K$ with group~$G$, considered up to topological ambient isotopy rel.  boundary?
\end{question}

Here, a group is called \emph{ribbon} if it arises as~$\pi_1(D^4 \setminus D)$ for some (smoothly embedded) ribbon disc~$D \subset D^4$ \footnote{$D \subset  D^4$ is \emph{ribbon} if the restriction of the radial function $D^4 \to \R$ to~$D$ is Morse and admits no local maxima.}.
We work with ribbon groups instead of fundamental groups of locally flat disc exteriors for convenience: the former admit an algebraic characterisation~\cite[Theorem~2.1]{FriedlTeichner},  while no such description appears to be known for the latter~\cite[Question 1.7]{FriedlTeichner}.
Examples of ribbon groups include~$G=\Z$ and the Baumslag-Solitar group~$G=BS(1,2)$ and in those cases, Question~\ref{question} has been fully resolved~\cite{FriedlTeichner,ConwayPowellDiscs}.
The answers, which will be partially recalled in Remark~\ref{rem:Classification} below,  both rely on Freedman's~$5$-dimensional s-cobordism theorem~\cite{Freedman} and therefore make use of the fact that~$\Z$ and~$BS(1,2)$ are \emph{good} groups.
We refer to~\cite[Definition 12.12]{DET} for the precise definition of a good group and to~\cite[Chapter 19]{DET} for a survey,  but note that the class of good groups contains all groups of subexponential growth as well as all elementary amenable groups (e.g.  solvable groups).
At the time of writing, it is unknown whether all groups are good: this  is equivalent to the question of whether the free group~$F_2$ is good~\cite[Proposition 19.7]{DET}.

\begin{remark}
\label{rem:Caveat}
The only elementary amenable ribbon groups are~$\Z$ and~$BS(1,2)$,  as can be seen by combining~\cite[Corollary 2.6.1]{Hillman} with the fact that ribbon groups have deficiency one and abelianise to $\Z$.
As a consequence,  if the class of good ribbon groups were eventually shown to coincide with the class of elementary amenable ribbon groups, 
then the current article would contain no new classification result.
On the other hand,  Theorem~\ref{thm:Injective} contains criteria for certain disc exteriors to be s-cobordant rel.\ boundary and holds regardless of the state of the art on the class of good groups. 
We also hope that the approach taken here will be of interest given the recent surge of activity around the topic of $2$-discs in the $4$-ball, both in the smooth and topological category~\cite{JuhaszZemke,ConwayPowellDiscs, Hayden,SundbergSwann,HaydenKjuchukovaKrishnaMillerPowellSunukjian,HaydenSundberg, HaydenComplex, LipshitzSarkar,DaiMallickStoffregen}.
\end{remark}

In order to give a flavour of our results without listing technical assumptions this early on, we mention a corollary of our main theorems (Theorems~\ref{thm:Injective} and~\ref{thm:Classification}).
To state this result succintly,  we introduce some terminology.
A \emph{$G$-ribbon disc} refers to a homotopy ribbon disc~$D \subset D^4$ with~$\pi_1(D^4 \setminus D) \cong G$, and given a knot $K$,  we write $\mathcal{D}_G(K)$ for the set of rel. \ boundary topological ambient isotopy classes of~$G$-ribbon discs with boundary $K$.
We also write~$M_K$ for the result of~$0$-surgery on~$K$ and use~$\operatorname{Epi}^{\operatorname{FT}}(\pi_1(M_K),G)$ to denote the set of epimorphisms $\pi_1(M_K) \twoheadrightarrow G$ that satisfy~\eqref{eq:FT} below.
While this definition will be discussed in greater detail in the next couple of sections, for the moment we simply note that $\operatorname{Aut}(G)$ acts on $\operatorname{Epi}^{\operatorname{FT}}(\pi_1(M_K),G)$ by postcomposition, allowing us to consider the orbit set~$\operatorname{Epi}^{\operatorname{FT}}(\pi_1(M_K),G)/\operatorname{Aut}(G)$.
Mapping a~$G$-ribbon disc $D \in \mathcal{D}_G(K)$ with aspherical complement to 
the inclusion induced homomorphism~$\pi_1(M_K) \twoheadrightarrow \pi_1(D^4 \setminus D)$  determines an element $\Phi(D)$ in this orbit set.

\begin{thmintro}
Fix a knot $K \subset S^3$.
\begin{enumerate}
\item If $G$ is a knot group (i.e.  $G=\pi_1(S^3 \setminus J)$ for some knot $J$), then exteriors of $G$-ribbon discs $D_1,D_2 \in \mathcal{D}_G(K)$ are s-cobordant rel. \ boundary if $\Phi(D_1)=\Phi(D_2)$.
If $G$ is good, then $\Phi$ induces a bijection $\mathcal{D}_{G}(K) \approx \operatorname{Epi}^{\operatorname{FT}}(\pi_1(M_K),G)/\operatorname{Aut}(G)$.
\item If $m,n \in \Z$ are such that $|m-n|=1$ and $G$ is the Baumslag-Solitar group $BS(m,n)=\langle a,b \mid ab^m a^{-1}=b^n \rangle$,  then exteriors of aspherical $G$-ribbon discs $D_1,D_2 \in \mathcal{D}_G(K)$ are s-cobordant rel. \ boundary if $\Phi(D_1)=\Phi(D_2)$.
If $G$ is good,  then $\Phi$ induces a bijection~$\mathcal{D}_G^a(K) \approx \operatorname{Epi}^{\operatorname{FT}}(\pi_1(M_K),G)/\operatorname{Aut}(G)$, where $\mathcal{D}_G^a(K) \subset \mathcal{D}_G(K)$ denotes the subset of~$G$-ribbon discs with aspherical exterior.
\end{enumerate}
\end{thmintro}

Examples~\ref{ex:KnotGroup} and~\ref{ex:BS} describe how this result follows from Theorems~\ref{thm:Injective} and~\ref{thm:Classification}.
Additionally,  as we explain in more detail in Remark~\ref{rem:Classification} below,  this theorem recovers the previously known classifications for the groups $BS(0,1)=\Z$ and $BS(1,2)$ since,  for these groups, homotopy-ribbon disc exteriors are known to be aspherical.


\subsection{Existence}

We recall and motivate a sufficient condition for the existence of a~$G$-ribbon disc with boundary $K$, which is due to Friedl and Teichner~\cite[Theorem 1.9]{FriedlTeichner}.
First,  if~$K$ bounds a locally flat disc~$D \subset D^4$, then $\partial N_D=M_K$, where $N_D:=D^4 \setminus \nu D$ is the exterior of~$D$ and~$M_K$ denotes the $3$-manifold obtained by $0$-framed surgery on~$K$.
Next, if~$D \subset D^4$ is a~$G$-ribbon disc for a knot~$K$, then there is an epimorphism~$\pi_1(M_K) \twoheadrightarrow \pi_1(N_D) \cong G$ and~$(N_D,M_K)$ satisfies Poincar\'e duality or, using surgery theory jargon,  is a (4-dimensional) \emph{Poincar\'e pair}.
If, additionally, the disc exterior~$N_D=D^4 \setminus \nu D$ is aspherical,  then we have a homotopy equivalence~$N_D  \simeq K(G,1)$ and we deduce that~$(K(G,1),M_K)$ is a Poincar\'e pair.

\begin{remark}
\label{rem:Whitehead}
It is expected that ribbon disc exteriors are aspherical~\cite[Conjecture 6.5]{GordonRibbon} (see also~\cite{HowieAsphericity}).
As noted in~\cite[Section~2]{FriedlTeichner} this 
would imply the \emph{ribbon group conjecture}: ribbon groups are geometrically $2$-dimensional\footnote{Friedl and Teichner refer to geometrically $2$-dimensional groups as \emph{aspherical} groups.}.
Here recall that a group~$G$ is called \emph{geometrically~$2$-dimensional} if~$K(G,1)$ is (homotopy equivalent to) a~$2$-complex.
Both statements are in fact particular cases of the~\emph{Whitehead conjecture} which states that every connected subcomplex of a~$2$-dimensional aspherical CW complex is itself aspherical~\cite{Whitehead}; see~\cite{Rosebrock} for a nice overview.
Howie proved that locally indicable ribbon groups are geometrically~$2$-dimensional~\cite[Theorem 5.2]{HowieLocally}.
On the other hand, to the best of our knowledge,  the Whitehead conjecture is not known to imply that exteriors of homotopy ribbon discs are aspherical; see also Remark~\ref{rem:Classification}~below.
\end{remark}


We argued that if $D$ is a $G$-ribbon disc with aspherical exterior and boundary a knot $K$,  then $\pi_1(M_K) \twoheadrightarrow \pi_1(N_D)\cong G$ is an epimorphism and $(K(G,1),M_K)$ is a Poincar\'e pair.
On the other hand, if we start with an epimorphism~$\pi_1(M_K)  \twoheadrightarrow G$ onto a group $G$,  then 
there is an embedding~$\varphi \colon  M_K \hookrightarrow K(G,1)=BG$ that induces the given surjection on fundamental groups and, if $G$ is geometrically $2$-dimensional, then~\cite[Lemma~3.2]{FriedlTeichner} shows that~$(K(G,1),M_K)$ is a Poincar\'e pair if and only if the induced map
\begin{equation}
\label{eq:FT}
 \varphi^* \colon H^i(BG;\Z[G]) \to H^i(M_K;\Z[G]_\varphi) \text{ is an isomorphism  for } i=1,2.
 \tag{FT}
 \end{equation}
Under an additional condition on the group~$G$, Friedl and Teichner prove that this leads to a sufficient condition for $K$ to bound a $G$-ribbon disc~\cite[Theorem 1.9 and Lemma 3.2]{FriedlTeichner}.

\begin{theorem}[Friedl-Teichner]
\label{thm:FT}
Let~$K \subset S^3$ be a knot and let~$G$ be a good geometrically $2$-dimensional ribbon group such that~$\widetilde{L}_4^h(\Z[G])=0$.
If~$\varphi \colon \pi_1(M_K) \twoheadrightarrow G$ is an epimorphism that satisfies~\eqref{eq:FT}, then there exists a~$G$-ribbon disc $D \subset D^4$ with aspherical exterior and boundary~$K$ 
such that the composition~$\pi_1(M_K) \twoheadrightarrow \pi_1(N_D)\cong G$ agrees with~$\varphi$.
\end{theorem}

\begin{remark}
\label{rem:FT}
We make a couple of remarks on this theorem.
\begin{itemize}
\item Friedl and Teichner actually prove a stronger result.  Instead of asking for~$G$ to be geometrically $2$-dimensional, they merely demand that~$H_3(G)=0$ and~$H^i(G;\Z[G])=0$ for~$i>2$ and instead of assuming that $G$ is ribbon, they only require that~$G$ be finitely presented  and satisfy~$H_1(G)=\Z$ and~$H_2(G)=0$.
Finally, they do not require $G$ to be good, only that the surgery sequence (with $h$-decorations) be exact for all 4-dimensional Poincar\'e pairs $(X,M)$ with~$\pi_1(X)=G$.
\item The fact that the disc exterior is aspherical is implicit in~\cite[proof of Theorem 1.9]{FriedlTeichner}: their surgery theoretic argument yields a disc $D$ whose exterior~$N_D=D^4 \setminus \nu D$ is homotopy equivalent to $K(G,1)$, which is aspherical.
\item The groups~$\Z$ and~$BS(1,2)$ satisfy all the assumptions of Theorem~\ref{thm:FT}.
Additionally, for those groups,  condition~\eqref{eq:FT} simplifies considerably.
Indeed if $G$ is poly-(torsion-free abelian) (or PTFA for short),  then~\eqref{eq:FT} reduces to
\begin{equation}
\label{eq:Ext}
\operatorname{Ext}_{\Z[G]}^1(H_1(M_K;\Z[G]_\varphi),\Z[G])=0
\tag{Ext}
\end{equation}
and for $G=\Z$ it reduces further to the condition $\Delta_K=1$; all of this is explained  in~\cite[Sections 1 and 4 and Lemma 3.3]{FriedlTeichner}.
\end{itemize}
\end{remark}

\subsection{Uniqueness and classification}

We now return to the set~$\mathcal{D}_G(K)$ of rel. \ boundary topological ambient isotopy classes of~$G$-ribbon discs with boundary~$K$.
In fact, we will mostly be concerned with the subset $\mathcal{D}^a_G(K) \subset \mathcal{D}_G(K)$ of discs with aspherical exteriors.
To that effect,  inspired by~\cite[Definition~1.2]{HambletonKreckTeichner}, we describe some assumptions on the group~$G$ that we will require.
\begin{definition}
\label{def:WAA}
A group~$G$ \emph{satisfies properties W-AA} if 
\begin{enumerate}
\item[\mylabel{item:W}{(W)}]the Whitehead group~$\operatorname{Wh}(G)$ vanishes; 
\item[\mylabel{item:A4}{(A4)}] the assembly map~$A_4 \colon  H_4(BG;\mathbf{L}_\bullet) \to L_4(\Z[G])$ is an isomorphism;\footnote{In the work of Hambleton, Kreck and Teichner~\cite{HambletonKreckTeichner} W-AA only requires $A_4$ to be injective.}
\item[\mylabel{item:A5}{(A5)}] the assembly map~$A_5 \colon H_5(BG;\mathbf{L}_\bullet) \to L_5(\Z[G])$ is surjective.
\end{enumerate}
\end{definition}
We will mostly use these conditions as a blackbox,  but note that thanks to extensive work on the Farrell-Jones conjecture (see~\cite{LueckIsomorphism} for a survey) they should not be thought of as insurmountable restrictions.
We discuss all of this in more detail in Remark~\ref{rem:Classification} below and refer to~\cite{RanickiTopological,ChangWeinberger,LueckAssembly,LueckIsomorphism} for background on assembly maps in $L$-theory.
Returning to our aim of describing~$\mathcal{D}_G(K)$,   we consider the set 
\begin{equation}
\label{eq:Epi}
 \operatorname{Epi}^{FT}(\pi_1(M_K),G):=\{ \varphi \colon \pi_1(M_K) \to G \ | \ \varphi \text{ is an epimorphism that satisfies \eqref{eq:FT}}  \}
  \tag{Epi}
 \end{equation}
and observe that it is acted upon (by postcomposition) by the group~$\operatorname{Aut}(G)$ of automorphisms of~$G$.
Thanks to the discussion leading up to Theorem~\ref{thm:FT}, note that sending a~$G$-ribbon disc with aspherical exterior to an epimorphism~$\pi_1(M_K) \stackrel{}{\twoheadrightarrow} \pi_1(N_D) \cong G$ defines a map
$$\Phi \colon \mathcal{D}_G^a(K) \to \operatorname{Epi}^{FT}(\pi_1(M_K),G)/\operatorname{Aut}(G)$$
which does not depend on the  choice of the isomorphism $\pi_1(N_D) \cong G$.
If $G$ is a good geometrically~$2$-dimensional ribbon group such that $\widetilde{L}_4(\Z[G])=0$,  then Theorem~\ref{thm:FT} ensures that $\Phi$ is surjective.
Our main technical result gives conditions on $G$ for $\Phi$ to be injective and,  in the absence of the goodness condition on $G$,  for exteriors of $G$-ribbon discs to be s-cobordant rel.\ boundary.


\begin{theorem}
\label{thm:Injective}
Let~$K$ be a knot and let~$G$ be a geometrically $2$-dimensional group that satisfies~\ref{item:W} and~\ref{item:A5}.
If $D_1$ and $D_2$ are two $G$-ribbon discs with aspherical exteriors and boundary~$K$ such that~$\Phi(D_1)=\Phi(D_2)$,  then
the disc exteriors $N_{D_1}$ and $N_{D_2}$ are s-cobordant rel.\ boundary.

If in addition to these conditions the group $G$ is good, then the discs $D_1$ and $D_2$ are ambiently isotopic rel.\ boundary.
\end{theorem}
\color{black}

We note that this result can alternatively be stated with normal subgroups instead of epimorphisms as this is easier to verify in practice.
To state this concisely,  given a slice disc $D$ for a knot~$K$, we use $\iota_D \colon \pi_1(M_K) \to \pi_1(N_D)$ to denote the inclusion induced map.


\begin{corollary}
\label{cor:Subgroup}
Let~$K$ be a knot and let~$G$ be a geometrically $2$-dimensional group that satisfies~\ref{item:W} and~\ref{item:A5}.
If $D_1$ and $D_2$ are two $G$-ribbon discs with aspherical exteriors and boundary~$K$ such that~$\ker(\iota_{D_1})=\ker(\iota_{D_2})$,  then the disc exteriors $N_{D_1}$ and $N_{D_2}$ are s-cobordant rel.\ boundary.

If in addition to these conditions the group $G$ is good, then the discs $D_1$ and $D_2$ are 
ambiently isotopic rel.\ boundary.
\end{corollary}

For smoothly embedded discs, the hypotheses of these results can be relaxed.
\begin{remark}
\label{rem:Geom2DSmooth}
If $D_1$ and $D_2$ are ribbon discs with aspherical exteriors and $\pi_1(N_{D_i}) \cong G$ for $i=1,2$,  then the assumption that~$G$ be geometrically $2$-dimensional can be omitted in both Theorem~\ref{thm:Injective} and Corollary~\ref{cor:Subgroup}: in this case $K(G,1) \simeq N_{D_i}$ has the homotopy type of a 2-complex.
\end{remark}

Combining Theorems~\ref{thm:FT} and~\ref{thm:Injective}, we obtain an answer to Question~\ref{question} provided we make some restrictions on the ribbon group $G$ and require the ribbon disc exteriors to be aspherical.

\begin{theorem}
\label{thm:Classification}
Let~$K \subset S^3$ be a knot and let~$G$ be a geometrically~$2$-dimensional good ribbon group that satisfies properties W-AA.
Mapping a~$G$-ribbon disc~$D$ to the epimorphism~$\pi_1(M_K) \stackrel{}{\twoheadrightarrow} \pi_1(N_D) \cong G$
defines a bijection $\Phi$ between the two following sets:
\begin{enumerate}
\item the set~$\mathcal{D}^a_G(K)$ of~$G$-ribbon discs  with aspherical exterior and boundary~$K$,  considered up to ambient isotopy rel.\ boundary;
\item the set~$\operatorname{Epi}^{\operatorname{FT}}(\pi_1(M_K),G)/\operatorname{Aut}(G)$ defined in~\eqref{eq:Epi}.
\end{enumerate}
\end{theorem}
\begin{proof}
We argue in Remark~\ref{rem:A4} below that since $G$ is a geometrically~$2$-dimensional ribbon group with $\operatorname{Wh}(G)=0$,
requiring~$G$ to satisfy condition~\ref{item:A4} is equivalent to asking for~$\widetilde{L}_4(\Z[G])=0$.
Thus the hypotheses of Theorem~\ref{thm:FT} are satisfied and so~$\Phi$ is surjective.
The injectivity of~$\Phi$ follows from Theorem~\ref{thm:Injective} which we can apply since~$G$ satisfies properties W-AA.
\end{proof}

\begin{remark}
\label{rem:Classification}
We collect a couple of remarks on this result.
\begin{itemize}
\item If the ribbon group conjecture (or more optimistically the Whitehead conjecture) were true, then requiring $G$ to be geometrically~$2$-dimensional would be superfluous; recall Remark~\ref{rem:Whitehead}.
It is also tempting to conjecture that exteriors of $G$-ribbon discs are aspherical and in this case we would have~$\mathcal{D}^a_G(K)=\mathcal{D}_G(K)$.
This latter conjecture holds when~$G$ is PTFA~\cite[Lemma 2.1]{ConwayPowellDiscs} (e.g. when $G=\Z$ and $G=BS(1,2)$) and is a consequence of the Whitehead conjecture if the disc exterior is homotopy equivalent to a $2$-complex.
%
\item The groups~$\Z$ and~$BS(1,2)$ satisfy the hypotheses of Theorem~\ref{thm:Classification} and in this case,  unpacking the definition of~$\operatorname{Epi}^{\operatorname{FT}}(\pi_1(M_K),G)/\operatorname{Aut}(G)$ recovers~\cite[Theorems~1.5 and~1.6]{ConwayPowellDiscs}.
Instead of repeating those statements,  we note that for~$G=\Z$,~$\operatorname{Epi}^{\operatorname{FT}}(\pi_1(M_K),G)/\operatorname{Aut}(G)$ has at most one element, while for~$G=BS(1,2)$ it has at most~$2$~\cite[Section 4]{ConwayPowellDiscs}.
Estimating the cardinality of this set in general appears to be more challenging.
Naturally,  the set~$\mathcal{D}_G(K)$ is often empty: for example, we refer to~\cite[Corollary 3.4]{FriedlTeichner} for an obstruction (based on the Alexander polynomial) to a knot $K$ bounding a $G$-ribbon disc. 
%
\item As we alluded to in Corollary~\ref{cor:Subgroup},  the classification result of Theorem~\ref{thm:Classification} can be stated in terms of normal subgroups of~$\pi_1(M_K)$ instead of epimorphisms originating from~$\pi_1(M_K)$: to a~$G$-ribbon disc $D$, one associates the normal subgroup $\ker(\pi_1(M_K) \twoheadrightarrow \pi_1(N_D))$ of $\pi_1(M_K)$.
This was the perspective taken in~\cite{ConwayPowellDiscs} where, using that~$BS(1,2)$ is metabelian,  the results were then formulated using submodules of the Alexander module~$H_1(M_K;\Z[t^{\pm 1}])$; the details are in~\cite[Section 3]{ConwayPowellDiscs}.
%
\item The requirement that the group be good is hard to verify in practice.
On the other hand~$G$ satisfies property W-AA if it is geometrically~$2$-dimensional and satisfies the Farrell-Jones conjecture: if a group~$G$ is geometrically~$2$-dimensional, then~$K(G,1)$ is a~$2$-complex and the claim now follows as in~\cite[Lemma 2.3]{KasprowskiLand} (the core of the argument will be recalled both in the proof of Theorem~\ref{thm:Injective} and in Remark~\ref{rem:A4}).
We treat the Farrell-Jones conjecture as a blackbox, but refer the interested reader to~\cite{LueckIsomorphism} for a survey and to~\cite[Chapter~15]{LueckIsomorphism} for a list of groups for which the conjecture is known to hold.
\end{itemize}
\end{remark}

\begin{example}
\label{ex:KnotGroup}
We argue that the group~$G(J)=\pi_1(S^3 \setminus J)$ of a classical knot~$J \subset S^3$ is a geometrically~$2$-dimensional ribbon group that satisfies properties W-AA.
Thus Theorem~\ref{thm:Injective} provides a criterion for exteriors of $G(J)$-ribbon discs to be s-cobordant rel.\ boundary and, if~$G(J)$ is additionally assumed to be good, then Theorem~\ref{thm:Classification} classifies~$G(J)$-ribbon discs for~$J \# -J$.

The group of~$J \subset S^3$ is ribbon (the ribbon knot~$J \# -J$ bounds a smoothly embedded  ribbon disc with group $G(J)$ as explained in~\cite[page 2135]{FriedlTeichner}).
The sphere theorem ensures that~$G(J)$ is geometrically~$2$-dimensional (the knot exterior is aspherical and has the homotopy type of a 2-complex; see e.g.~\cite[Theorem 11.7]{LickorishIntroduction}).
The Farrell-Jones conjecture holds for~$G(J)$ because it holds for the fundamental group of any~$3$-manifold with boundary~\cite[Theorem~15.1~(e)]{LueckIsomorphism}.

Since knot groups are PTFA by work of Strebel~\cite{Strebel},~$G(J)$-ribbon discs are aspherical by~\cite[Lemma~2.1]{ConwayPowellDiscs} and thus~$\mathcal{D}_{G(J)}(J \# -J)=\mathcal{D}^a_{G(J)}(J \# -J)$.
Finally,  as we noted in Remark~\ref{rem:FT}, since $G(J)$ is PTFA, we can use condition~\eqref{eq:Ext} instead of condition~\eqref{eq:FT}.
\end{example}
\begin{example}
\label{ex:BS}
We argue that for~$m,n \in \Z$ with~$|m-n|=1$, the Baumslag-Solitar group~$BS(m,n)$  is a geometrically~$2$-dimensional ribbon group that satisfies properties W-AA.
Thus Theorem~\ref{thm:Injective} provides a criterion for exteriors of aspherical $BS(m,n)$-ribbon discs to be s-cobordant rel.\ boundary and, if~$BS(m,n)$ is additionally assumed to be good, then Theorem~\ref{thm:Classification} classifies~$BS(m,n)$-ribbon discs with aspherical exteriors.


The fact that~$BS(m,n)$ is ribbon when~$|m-n|=1$ can be seen by looking at the handle diagram depicted in Figure~\ref{fig:RibbonBS}.
Baumslag-Solitar groups are geometrically~$2$-dimensional: the universal cover of the presentation~$2$-complex for~$\langle a,b \ | \ ba^{m}b^{-1}=b^n \rangle$ is homeomorphic to the product of~$\R$ with a tree; see e.g.~\cite[Section~2]{FredenKnudsonSchofield}.
Additionally, every Baumslag-Solitar group~$BS(m,n)$ satisfies the Farrell-Jones conjecture~\cite{FarrellWu,GandiniMeinertRuping}.
\end{example}

\begin{figure}[!htbp]
\centering
\includegraphics[scale=0.5]{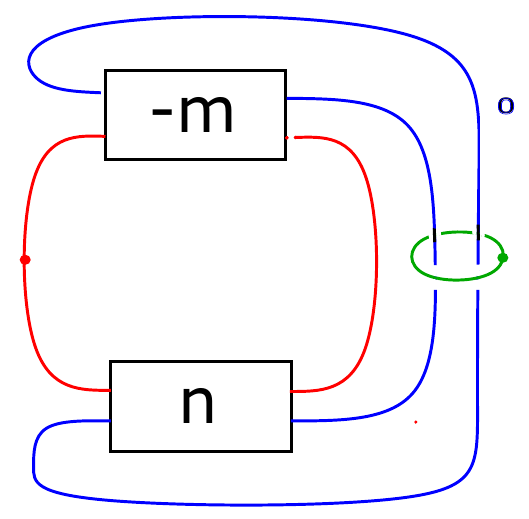}
\caption{Assuming that~$|m-n|=1$,  this figure depicts a handle diagram of a ribbon disc exterior with fundamental for~$BS(m,n)$.
Indeed, since~$|m-n|=1$, the red and blue knots form a handle diagram for~$D^4$ in which the green knot is sliced by a ribbon disc~$D$ with~$\pi_1(N_D)=BS(m,n)$.}
\label{fig:RibbonBS}
\end{figure}

We conclude with a brief final remark concerning asphericity.
The methods of this paper rely heavily on~$G$-ribbon disc exteriors (conjecturally) being aspherical.
Currently, non-aspherical~$4$-manifolds with boundary~$M_K$ and fundamental group~$G$ are poorly understood beyond the group~$G=\Z$~\cite{ConwayPowell}.
This is the reason why we only work in~$D^4$ instead of in other~$4$-manifolds.

\subsection*{Acknowledgments}
I wish to thank Daniel Kasprowski and Markus Land for insightful correspondence related to~\cite{KasprowskiLand} and for helpful comments on a draft of this paper.
I am also grateful to Lisa Piccirillo for explaining to me why~$BS(m,n)$ is ribbon when~$|m-n|=1$ and to Jonathan Hillman for pointing me towards~\cite[Corollary 2.6.1]{Hillman}.
Finally,  thanks also go to anonymous referees for helpful comments and suggestions.

\subsection*{Conventions}
Throughout this article, we work in the topological category. 
Manifolds are assumed to be compact and oriented.
Homeomorphisms, homotopy equivalences and isotopies are \emph{rel.\ boundary} if they fix the boundary pointwise.
If~$M_1,M_2$ are two~$n$-manifolds with boundary~$Y$, a cobordism between~$M_1$ and~$M_2$ is \emph{relative~$Y$} if, when restricted to~$Y$, it is the product~$Y \times [0,1]$.

\section{Proof of the main technical result}

We recall the statement of Theorem~\ref{thm:Injective} and prove it.
Let~$K$ be a knot and let~$G$ be a geometrically $2$-dimensional 
group that satisfies~\ref{item:W} and~\ref{item:A5}.
The aim is to prove that if~$D_1$ and~$D_2$ are two~$G$-ribbon discs with aspherical exteriors and boundary $K$ such that $\Phi(D_1)=\Phi(D_2) \in \operatorname{Epi}^{FT}(\pi_1(M_K),G)/\operatorname{Aut}(G)$, then the disc exteriors $N_{D_1}$ and $N_{D_2}$ are s-cobordant rel. boundary and, if $G$ is additionally assumed to be good, then~$D_1$ and~$D_2$ are 
ambiently isotopic rel.\ boundary.

\begin{proof}[Proof of Theorem~\ref{thm:Injective}]
Assume that~$D_1$ and~$D_2$ are two~$G$-ribbon discs with aspherical exteriors and boundary~$K$ and that their epimorphisms agree in~$\operatorname{Epi}^{FT}(\pi_1(M_K),G)/\operatorname{Aut}(G)$.
We must show that the exteriors~$N_{D_1}$ and~$N_{D_2}$ are  s-cobordant rel.\ boundary. 
If we additionally assume that $G$ is good,  then Freedman's~$5$-dimensional relative~$s$-cobordism theorem will then ensure that~$N_{D_1}$ and~$N_{D_2}$ are in fact homeomorphic rel.\ boundary. 
The fact that~$D_1$ and~$D_2$ are ambiently isotopic rel.\ boundary follows by applying Alexander's trick, as noted in~\cite[Lemma 2.5]{ConwayPowellDiscs}.
Our strategy decomposes into two steps.
The first step uses the conditions on the epimorphisms to show that~$\id_{M_K}$ extends to a homotopy equivalence~$N_{D_1} \simeq N_{D_2}$.
The second step uses surgery theory to improve this homotopy equivalence to
an s-cobordism
rel.\ boundary; here is where we rely on properties~\ref{item:W} and~\ref{item:A5} as well as on the fact that~$G$ is good.

We start with the first step.
Since the epimorphisms of~$D_1$ and~$D_2$ agree,  there exists an automorphism~$\Psi$ of~$G$ that makes the following diagram commute:
$$
\xymatrix@R0.5cm{
\pi_1(M_K) \ar@{->>}[d]_{\iota_{D_1}}\ar[r]^=& \pi_1(M_K)\ar@{->>}[d]^{\iota_{D_2}} \\
\pi_1(N_{D_1})\ar[d]_\cong & \pi_1(N_{D_2}) \ar[d]^\cong \\
G \ar[r]^-{\Psi,\cong}&G.
}
$$
Since the bottom vertical maps in this diagram are isomorphisms,  we deduce that there exists an isomorphism~$g \colon \pi_1(N_{D_1})\cong \pi_1(N_{D_2})$ such that~$g \circ \iota_{D_1}  =\iota_{D_2}$; such isomorphisms were called  \emph{compatible} in~\cite[Section 2]{ConwayPowellDiscs}.
As the $D_i$ have aspherical exteriors, the obstruction theory argument from~\cite[end of proof of Lemma 2.1]{ConwayPowellDiscs} shows that the identity~$\id_{M_K} \colon M_K \to M_K$ extends to a homotopy equivalence~$f \colon N_{D_1} \to N_{D_2}$ which induces~$g$ on fundamental groups.

We now move on to the second step: we use surgery theory to improve the homotopy equivalence~$f$ to an s-cobordism~$N_{D_1} \cong_{\operatorname{s-cob}} N_{D_2}$ rel.\ boundary.
We describe the argument very briefly for readers that are familiar with surgery theory before giving some more details.
Consider the surgery sequence,  where we can ignore decorations thanks to condition~\ref{item:W}:
$$  \mathcal{N}(N_{D_2} \times [0,1],\partial (N_{D_2} \times [0,1])) \xrightarrow{\sigma_5} L_{5}(\Z[G]) \to \mathcal{S}(N_{D_2},\partial N_{D_2}) \xrightarrow{\eta} \mathcal{N}(N_{D_2},\partial N_{D_2}) \xrightarrow{\sigma_4} L_4(\Z[G]).$$
We use that disc exteriors have trivial~$H_2$ to deduce that $\eta$ is the zero map.
More concretely,  
we obtain a degree one normal map
\begin{equation}
\label{eq:2ConnectedDONM}
(F',f,\id_{N_{D_2}}) \colon (W',N_{D_1},N_{D_2}) \to (N_{D_2} \times [0,1],N_{D_2},N_{D_2})
\end{equation}
that we can assume to be~$2$-connected by surgery below the middle dimension.
We then use property~\ref{item:A5} and the fact that $G$ is geometrically $2$-dimensional to deduce that~$\sigma_5$ is surjective.
We infer that~$N_{D_1}$ and~$N_{D_2}$ are~$s$-cobordant either by appealing to the exactness of the surgery sequence (which requires $G$ to be good) or by using the surjectivity of~$\sigma_5$ to
replace~$F'$ by another degree one normal map with vanishing surgery obstruction (despite being slightly longer, this argument has the advantage of not requiring $G$ to be good).
Thus the fact that $N_{D_1}$ and $N_{D_2}$ are $s$-cobordant rel.\ boundary can be proved without using that $G$ is good.
The homeomorphism classification result then follows from Freedman's~$5$-dimensional relative~$s$-cobordism theorem which we can apply 
if~$G$ is good.
%

We give more details.
The set~$\mathcal{N}(N_{D_2},\partial N_{D_2})$ consists of equivalences classes of degree one normal maps~$M \to N_{D_2}$ that restrict to a homeomorphism on the boundary.
Two such degree one normal maps~$f_i \colon  M_i \to N_{D_2}$ for~$i=1,2$ are equivalent if there exists a rel.\ boundary cobordism~$(W,M_1,M_2)$ and a degree one normal map 
$$(W,M_1,M_2) \to (N_{D_2} \times [0,1],N_{D_2},N_{D_2})$$ 
that restricts to~$f_i$ on~$M_i$ for~$i=1,2$.
A homotopy equivalence~$h \colon M \to N_{D_2}$ rel.\ boundary is in particular a degree one normal map that we denote by~$\eta(h) \in \mathcal{N}(N_{D_2},\partial N_{D_2})$.

We claim that $\eta$ is the zero map.
Under the isomorphism
\begin{equation}
\label{eq:NormalInvariants}
\mathcal{N}(N_{D_2},\partial N_{D_2}) \cong H^4(N_{D_2},\partial N_{D_2}) \oplus H^2(N_{D_2},\partial N_{D_2};\Z_2)=H^4(N_{D_2},\partial N_{D_2}) \cong \Z
\end{equation}
we have~$\eta(h)=\frac{1}{8}(\sigma(M)-\sigma(N_{D_2})$; this fact is well known to surgeons but we refer to~\cite[Proposition  2.2]{ConwayPowellDiscs} in case the reader is curious about the details.
Since the signature of a disc exterior vanishes and $h$ is a homotopy equivalence,  we deduce that $\eta(h)=0$, as claimed.



We assert that the map~$\sigma_5$ from the surgery  sequence is surjective.
This relies on surgery spectra and the algebraic theory of surgery. We treat this largely as a blackbox but note that this part of surgery theory was developed by Quinn~\cite{QuinnGeometricFormulation,QuinnSurgeryObstruction} and Ranicki~\cite{RanickiTotal,RanickiExact}; we also refer to~\cite[Section 4.4]{ChangWeinberger} for a nice overview of these topics and to~\cite[Section 4]{CenceljMuranovRepovs} for a helpful account of the rel.\ boundary case.
Using the relation between the assembly map and the surgery  obstruction (as mentioned for example in~\cite[page 158]{ChangWeinberger}) and the fact that~$N_{D_2}$ is a~$K(G,1)$, the following diagram commutes:
$$ 
\xymatrix@R0.5cm{
 \mathcal{N}(N_{D_2} \times [0,1],\partial (N_{D_2} \times [0,1])) \ar[rr]^-{\sigma_5}\ar[d]^\cong&& L_{5}(\Z[G])\ar[d]^= \\
H_5(N_{D_2};\mathbf{L}\langle 1 \rangle_\bullet) \ar[r]^\cong \ar[d]^\cong&H_5(N_{D_2};\mathbf{L}_\bullet) \ar[r]^-{}\ar[d]^\cong& L_5(\Z[G])\ar[d]^=\\  
H_5(BG;\mathbf{L}\langle 1 \rangle_\bullet) \ar[r]^\cong&H_5(BG;\mathbf{L}_\bullet) \ar[r]^-{A_5}& L_5(\Z[G]).
  }
$$
Here~$\mathbf{L}_\bullet$ denotes the~$L$-theory spectrum of the  integers and~$\mathbf{L}\langle 1 \rangle_\bullet$ denotes its~$1$-connective cover.
The fact that~$H_5(BG;\mathbf{L}\langle 1 \rangle_\bullet) \to H_5(BG;\mathbf{L}_\bullet)$ is an isomorphism follows because~$K(G,1)$ admits a~$2$-dimensional CW-model (the Atiyah-Hirzebruch spectral sequence argument is the same as in~\cite[proof of Lemma 2.3]{KasprowskiLand}) and the fact that the top left vertical map is an isomorphism is a fact from algebraic surgery theory; see e.g.~\cite[Equation (27)]{CenceljMuranovRepovs}.
Using this commutative diagram and property~\ref{item:A5} (which stipulates that the assembly map~$A_5$ is surjective), one deduces that~$\sigma_5$ is surjective.

There are now two closely related ways to conclude that $N_{D_1}$ and~$N_{D_2}$ are s-cobordant rel.\ boundary.
The first way is shorter but uses that the group $G$ is good: since $\eta$ is the zero map,
$\sigma_5$ is surjective and the surgery sequence is exact (because $G$ is good), the structure set~$\mathcal{S}(N_{D_2},\partial N_{D_2})$ (to which $f$ belongs) is trivial.
The second argument (inspired by~\cite{KasprowskiLand}) is slightly longer but does not require that the group $G$ be good:
since $\eta \equiv 0$,  there is a rel.\ boundary cobordism~$(W,N_{D_1},N_{D_2})$ and a degree one normal map 
$$(F,f,\id_{N_{D_2}}) \colon (W,N_{D_1},N_{D_2}) \to (N_{D_2} \times [0,1],N_{D_2},N_{D_2}).$$ 
Perform surgery below the middle dimension on the interior of~$W$ to obtain the~$2$-connected degree one normal map~$F'$ with surgery obstruction~$x:=\sigma(F') \in L_5(\Z[G])$ that we alluded to in~\eqref{eq:2ConnectedDONM}.
Using
the surjectivity of~$\sigma_5$,  one can find a degree one normal map 
$$\Psi  \colon (V,N_{D_2},N_{D_2}) \to (N_{D_2} \times [0,1],N_{D_2},N_{D_2})$$
 that restricts to the identity on both boundary components and with~$-x$ as its surgery obstruction;
stacking~$\Psi$ on top of~$F'$ leads to a degree one normal map~$F''$ with vanishing surgery obstruction~$\sigma(F'') \in L_5(\Z[G])$ and it follows that~$F''$ is normal bordant rel.\ $M_K \times [0,1]$ to a homotopy equivalence.
Thus, we have two arguments for why~$N_{D_1}$ and~$N_{D_2}$ are s-cobordant rel.\ boundary.
%

If~$G$ is good, we can apply Freedman's~$5$-dimensional relative~$s$-cobordism theorem~\cite[Theorem 7.1A]{FreedmanQuinn} and it follows that~$N_{D_1}$ and~$N_{D_2}$ are homeomorphic rel.\ boundary.
As we already mentioned,~\cite[Lemma~2.5]{ConwayPowellDiscs} implies that the discs are ambiently isotopic rel.\ boundary.
\end{proof}

We conclude by proving a statement that was used in the proof of Theorem~\ref{thm:Classification}.

\begin{remark}
\label{rem:A4}
Assume that~$G$ is a geometrically~$2$-dimensional ribbon group with vanishing Whitehead torsion (condition~\ref{item:W}).
We claim that~$G$ satisfies~$\widetilde{L}_4(\Z[G])=0$ if and only if it satisfies~\ref{item:A4},  which stipulates that the assembly map~$A_4 \colon  H_4(BG;\mathbf{L}_\bullet) \to L_4(\Z[G])$ is an isomorphism.
Since~$G$ is a ribbon group, there is a (smoothly embedded) ribbon disc~$D \subset D^4$ with~$\pi_1(N_D) \cong G$.
This time, $N_D$ might not be aspherical, but it is still a $2$-complex with vanishing $H_2$.
An Atiyah-Hirzebruch spectral sequence argument therefore shows that $H_4(N_D;\mathbf{L}\langle 1 \rangle_\bullet) \to H_4(BG;\mathbf{L}\langle 1 \rangle_\bullet)$ is an isomorphism.
Here, it is helpful to note that $H_2(G)=0$: use~$H_2(N_D)=0$ together with the exact sequence $\pi_2(N_D) \to H_2(N_D) \to H_2(\pi_1(N_D)) \to 0$;  see e.g.~\cite[Equation (0.1)]{Brown}.
The same argument as above then produces the following commutative diagram:
$$ 
\xymatrix@R0.5cm{
 \mathcal{N}(N_{D},M_K) \ar[rr]^-{\sigma_4}\ar[d]^\cong&& L_{4}(\Z[G])\ar[d]^= \\
H_4(N_{D};\mathbf{L}\langle 1 \rangle_\bullet) \ar[r]^\cong \ar[d]^\cong&H_4(N_{D};\mathbf{L}_\bullet) \ar[r]^-{}\ar[d]^\cong& L_4(\Z[G])\ar[d]^=\\  
H_4(BG;\mathbf{L}\langle 1 \rangle_\bullet) \ar[r]^\cong&H_4(BG;\mathbf{L}_\bullet) \ar[r]^-{A_4}& L_4(\Z[G]).
  }
$$
As explained in~\eqref{eq:NormalInvariants} and~\cite[Section 11.3B]{FreedmanQuinn},  the surgery obtruction~$\sigma_4$ maps the set of  normal invariants~$ \mathcal{N}(N_{D},\partial N_D)\cong \Z$ isomorphically onto the~$L_4(\Z)\cong \Z$-summand of~$L_4(\Z[G])=L_4(\Z) \oplus \widetilde{L}_4(\Z[G]).$
The claim now follows by combining this fact with the commutativity of the diagram.
\end{remark}

\bibliographystyle{alpha}
\bibliography{BiblioGribbon}
\end{document}